\documentclass[journal]{IEEEtran}

\usepackage{graphicx}  

\usepackage{subfigure} 

\usepackage{amsmath}   

\usepackage{amssymb}

\usepackage{rotating}

\newtheorem{theorem}{Theorem}

\newtheorem{definition}{Definition}

\newtheorem{proposition}{Proposition}

\def\x{{\mathbf x}}
\def\y{{\mathbf y}}
\def\u{{\mathbf u}}

\def\C{{\mathbf C}}
\def\D{{\mathbf D}}

\def\a{{\mathbf a}}

\def\g{{\mathbf g}}
\def\n{{\mathbf n}}
\def\H{{\mathbf H}}
\def\M{{\mathbf M}}
\def\I{{\mathbf I}}

\def\g{{\mathbf g}}

\def\t{{\mathbf t}}
\def\s{{\mathbf s}}

\begin{document}

\title{MMSE Estimation Under \\ Gaussian Mixture Statistics}
\title{Minimum Mean Square Error Estimation \\ Under Gaussian Mixture Statistics}

\author{John T. Fl{\aa}m, Saikat Chatterjee, Kimmo Kansanen, Torbj\"{o}rn Ekman
\thanks{John T. Fl{\aa}m, Kimmo Kansanen and Torbj\"{o}rn Ekman are with the Department of Electronics and Telecommunications, NTNU-Norwegian University of Science and Technology, Trondheim, Norway. Emails: flam@iet.ntnu.no, kimmo.kansanen@iet.ntnu.no and torbjorn.ekman@iet.ntnu.no. Saikat Chatterjee is with the Communication Theory Lab, School of Electrical Engineering, KTH-Royal Institute of Technology, Sweden. Email: saikatchatt@gmail.com, sach@kth.se.}}

\maketitle

\begin{abstract}
This paper investigates the minimum mean square error (MMSE) estimation of $\x$, given the observation $\y=\H\x+\n$, when $\x$ and $\n$ are independent and Gaussian Mixture (GM) distributed. The introduction of GM distributions, represents a generalization of the more familiar and simpler Gaussian signal and Gaussian noise instance.
We present the necessary theoretical foundation and derive the MMSE estimator for $\x$ in a closed form. Furthermore, we provide upper and lower bounds for its mean square error (MSE). These bounds are validated through Monte Carlo simulations.
      
\end{abstract}


%
\IEEEpeerreviewmaketitle

\section{Introduction}

In estimation theory, an important model is the Bayesian linear model
\begin{equation}
\mathbf{y=Hx+n},
\label{linmod}
\end{equation}
where $\y$ is a vector of observations, $\mathbf{H}$ is a known matrix, $\x$ is the vector to be estimated and $\n$ is additive noise. If $\x$ and $\n$ are mutually independent Gaussian variates, then the \textit{minimum mean square error} (MMSE) estimator for $\x$ is well known and quite tractable, see e.g. \cite{151045}. 

There are, however, often good reasons to go beyond the Gaussian setting. For one, $\x$ and $\n$ may not be Gaussian. For another, the distributions of $\x$ and $\n$ may even be multi modal. For these reasons, besides some appreciation of greater
generality, the pure Gaussian perspective is relaxed in this paper.

The extension, considered below, maintains independence between ${\mathbf{x}}$
and ${\mathbf{n}}$, but now either vector variate originates from a finite \textit{Gaussian mixture} (GM) distribution. Specifically, 
\begin{equation}
\x\sim \sum_{k\in\mathcal{K}}p_k \mathcal{N}(\u^{(k)}_{\x},\C^{(k)}_{\x\x}) \text{   and   } \n\sim \sum_{l\in\mathcal{L}}q_l \mathcal{N}(\u^{(l)}_{\n},\C^{(l)}_{\n\n}),
\label{m_mix}
\end{equation}
where the notation should be read in the distributional sense: $\x$ originates, with  a prior probability $p_k$, from a Gaussian source with distribution law $\mathcal{N}(\u^{(k)}_{\x},\C^{(k)}_{\x\x})$. Naturally, we require $\sum_{k} p_k =1$ and $p_k \geq 0$. 
The noise, $\n$, emerges in a similar but independent manner. $\mathcal{K}$ and
$\mathcal{L}$ are finite index sets. Their cardinalities determine the number of Gaussian \textit{components} in the mixtures. Clearly, when $\mathcal{K}$ and $\mathcal{L}$ are singletons,
we fall back on the familiar case of Gaussian signal and Gaussian noise. 
The component probabilities, component means and component covariances $\left(p_k, \u^{(k)}_{\x} \text{ and }\C^{(k)}_{\x\x}\right)$ are collectively referred to as the \textit{parameters} of a Gaussian mixture.

Several properties speak in favor of GM distributions. An important one is that a GM distribution can, in theory, approximate any distribution with arbitrary accuracy. Said differently, the closure of GM distributions on the vector space $\mathbb{X}$ is the set \textit{all} probability distributions on $\mathbb{X}$.
Thus, for any random vector $\mathbf{\x}\in\mathbb{X}$ there
exists a sequence of random variables $\mathbf{x}_{n}$, all of which are GM distributed, such that
\begin{eqnarray}
\begin{array}{l}
\lim_{n\rightarrow\infty}E\left\{g(\mathbf{x}_{n})\right\}=E\left\{g(\mathbf{\x})\right\} {\text{ for any bounded,}} \\ 
{\text{ continuous function }} g:\mathbb{X\rightarrow R}. 
\end{array} \nonumber
\end{eqnarray}
Therefore, by judiciously choosing the number of components, $\left|  \mathcal{K}\right|$, and the corresponding parameters, the underlying input $\mathbf{x}$ is approximated
``in distribution'' as closely as desired by a Gaussian mixture. For a formal argument see e.g. \cite{Sorenson1971465}. The intuition behind this asymptotic behavior is straightforward. First,
$\mathbf{x}$ can be approximated ad libitum by a mixture (a
convex combination) of Dirac measures. Second, each Dirac point measure is
approximated by a normal distribution having that point as its mean - and
a small covariance\footnote{Approximating an arbitrary distribution by a GM distribution, is generally a non-trivial problem.
This paper is, however, \textbf{\textit{not}} about density approximation/learning. Here we assume that $\x$ and $\n$ are associated with known GM distributions. Whether these distributions are exact or approximations is not the focus here. }. 


A second reason for using GM distributions on $\x$ and $\n$ in (\ref{linmod}), is that this produces a posterior distribution on $\x|\y$ which is also a GM. An analytic posterior distribution is very attractive: it quantifies our degree of belief in $\x$ for any $\y$, and any optimal Bayesian estimator (with respect to any criterion) may be derived from it. 

Last, but not least, it is easy to calculate the mean and covariance
of mixture distributions. These crucial parameters are transferred from
underlying components in convenient ways. So, to the extent that first- and
second-order analysis is important (the MMSE estimator corresponds to the
posterior mean), mixtures have a lot to offer.

Admittedly, to pass from from a pure Gaussian model to a corresponding GM model
is not without challenges and drawbacks. A notable one, as we shall see, is that mean square error of the MMSE estimator cannot be determined analytically.

There exists some related work on this topic. In \cite{928914}, \cite{4808405} and \cite{1597257}, it is shown that if two vectors $\x$ and $\y$ are jointly GM distributed, then the conditional distribution for $\x|\y$ is also a GM. These works do, however, not explicitly assume that $\y$ and $\x$ are related through a linear model, like (\ref{linmod}). 
In \cite{Flamicassp}, \cite{4518754}, \cite{kundusaikat}, \cite{kundusaikatsr}, linear models are assumed. In all of these works, $\x$ is
a GM, whereas $\n$ is purely Gaussian. For that simpler instance, the analytic MMSE estimator for $\x$ is provided. In \cite {1628628}, recursive estimation of a GM distributed state sequence from GM distributed measurements is considered. The resulting optimal estimator is termed a non Gaussian Kalman Filter. 

The above mentioned related works have three aspects in common - all of which invite for further investigations: (i) they all assume that the observation noise is purely Gaussian (which we believe is only a special case of GM noise), (ii) the theoretical foundation upon which the presented estimators rest is not explicitly presented, and most importantly (iii), proper analysis of the resulting \textit{mean square error} (MSE) is completely absent. 
For these reasons, a unified exposition including the derivation of the MMSE estimator for GM input and GM noise, its theoretical foundation, and analysis of its MSE, deserves to be made explicit.
To the best of our knowledge, none exists in the literature. 

In the next section, we present a theorem which compactly presents the main result of the paper: the analytical MMSE estimator with upper and lower performance bounds. In section \ref{sec:Analytical_MMSE_Estimator}, we derive the
posterior distribution rigorously, relying on the theory provided by the appendix. From the posterior, the MMSE estimator follows naturally. This proves the first part of the theorem. Section \ref{sec:Error_Analysis_of_the_MMSE_Estimator} analyzes the MSE of the MMSE estimator when the posterior is a GM, and shows that the MSE cannot be determined in a closed analytic form. Instead, we derive  upper and lower bounds for the MSE, which proves the second part of the theorem. In section \ref{sec:Simulation_Results}, these bounds are validated through Monte Carlo simulations, followed by the conclusion in section~\ref{sec:Conclusions}.

\section{The MMSE estimator with performance bounds}\label{sec:Main Result}
\begin{theorem}\label{theo:GM_estimator}
If the data are described by the Bayesian linear model (\ref{linmod}) where $\H$ is a known matrix, and $\x$ and $\n$ are independent and GM distributed as in (\ref{m_mix}), then the MMSE estimator of $\x$ is 
\begin{eqnarray}
\begin{array}{rcl}
\hat{\x} & = & \displaystyle\sum_{k,l} \alpha^{(k,l)}(\y) \,\, \left[ \u^{(k)}_{\x}+\C^{(k)}_{\x\x}\H^T \right. \\ 
& & \left. \left(\H\C^{(k)}_{\x\x}\H^T +\C^{(l)}_{\n\n}\right)^{-1}  \left(\y-\H\u^{(k)}_{\x} -\u^{(l)}_{\n}\right) \right]
\end{array}
\label{eq:GM_Density_MMSE_Estimator}
\end{eqnarray}
where
\begin{align}
\alpha^{(k,l)}(\y)=\frac{p_k q_l f^{(k,l)}(\y)}{\sum_{r,s}p_r q_s f^{(r,s)}(\y)},\nonumber
\end{align}
and $f^{(k,l)}(\y)$ is a Gaussian \textit{probability density function} (PDF) in $\y$ with mean
 \begin{align}
 \u^{(k,l)}_{\y}=\H\u^{(k)}_{\x} +\u^{(l)}_{\n},\nonumber
 \end{align}
and covariance
 \begin{align}
 \C^{(k,l)}_{\y\y}=\H\C^{(k)}_{\x\x}\H^T +\C^{(l)}_{\n\n} \nonumber.
 \end{align}

The performance of the MMSE estimator, measured by its MSE, $\epsilon^2 = E \left\{ \| \x -\hat{\x} \|_{2}^{2}\right\}$, is lower and upper bounded by 
\begin{align}
&\sum_{k,l} p_k q_l \text{Tr}\left( \C^{(k)}_{\x\x} - \C^{(k)}_{\x\x} \H^T  \left( \H\C^{(k)}_{\x\x}\H^{T}\hspace{-0.1cm}+ \C^{(l)}_{\n\n} \right)^{-1} \hspace{-0.1cm} \H \C^{(k)}_{\x\x} \right)\nonumber \\ 
&\leq  \epsilon^2 \label{eq:GM_Density_MMSE_Error_Bounds} \\
&\leq  \text{Tr} \left( \C_{\x\x} -  \C_{\x\x}  \H^{T} \left( \H\C_{\x\x}\H^{T}  +  \C_{\n\n} \right)^{-1}  \H\C_{\x\x}  \right).\nonumber
\end{align}
In (\ref{eq:GM_Density_MMSE_Error_Bounds}), $\text{Tr}(\cdot)$ denotes the trace operator, and
\begin{align}
\C_{\x\x} &=\sum_{k} p_k \left(\C^{(k)}_{\x\x}+\u^{(k)}_{\x}{\u^{(k)}_{\x}}^{T}\right)-\u_{\x}\u^{T}_{\x} \label{covx},\\
\u_{\x}&=\sum_{k} p_k \u^{(k)}_{\x}\label{mx},\\
\C_{\n\n} &=\sum_{l} q_l \left(\C^{(l)}_{\n\n}+\u^{(l)}_{\n}{\u^{(l)}_{\n}}^{T}\right)-\u_{\n}\u^{T}_{\n} \label{covn},\\
\u_{\n}&=\sum_{l} q_l \u^{(l)}_{\n}\label{mn}. 
\end{align}
\end{theorem}
The proof of (\ref{eq:GM_Density_MMSE_Estimator}) is given in section \ref{sec:Analytical_MMSE_Estimator}, whereas the proof of (\ref{eq:GM_Density_MMSE_Error_Bounds}) is given in section \ref{sec:Error_Analysis_of_the_MMSE_Estimator}.

\section{Deriving the analytical MMSE Estimator}
\label{sec:Analytical_MMSE_Estimator}

Our assumption is that $\x$ and $\n$ are independent and GM distributed as in (\ref{m_mix}). Then, by Proposition \ref{Joint distribution of independent GM distributed random vectors} from the appendix, $\x$ and $\n$ are jointly GM distributed as 
\begin{align}
 \left [ \begin{array}{c} \x \\ \n \end{array} \right ] \sim \sum_{k, l}p_k q_l\mathcal{N}\left(\left [ \begin{array}{c} \u^{(k)}_{\x} \\ \u^{(l)}_{\n} \end{array} \right ],\left [ \begin{array}{cc} \C^{(k)}_{\x\x} & 0 \\ 0 & \C^{(l)}_{\n\n} \end{array} \right ]\right).\nonumber
 \end{align} 
Observe that equation (\ref{linmod}) can be written as
\begin{align}
\left [ \begin{array}{c} \y \\ \x \end{array} \right ]=\left [ \begin{array}{cc} \H & \I \\ \I & 0 \end{array} \right ]\left [ \begin{array}{c} \x \\ \n \end{array} \right ].\nonumber
 \end{align}
Therefore, the joint vector $[\y^{T} \,\, \x^{T}]^{T}$ is a linear transform of the GM distributed vector $[\x^{T} \,\, \n^{T}]^{T}$. By Proposition \ref{Affine transform of a GM distributed random vector.} of the appendix, the joint vector $[\y^{T} \,\, \x^{T}]^{T}$ is GM distributed as well:
\begin{eqnarray}
\begin{array}{r}
\left [ \begin{array}{c} \y \\ \x \end{array} \right ]
\sim \sum_{k,l}p_k q_l\mathcal{N}\left(\left [ \begin{array}{c} \H\u^{(k)}_{\x} +\u^{(l)}_{\n}\\ \u^{(k)}_{\x} \end{array} \right ], \right. \\
\left. \left [ \begin{array}{cc} \H\C^{(k)}_{\x\x}\H^T +\C^{(l)}_{\n\n}& \H\C^{(k)}_{\x\x} \\ \C^{(k)}_{\x\x}\H^T & \C^{(k)}_{\x\x} \end{array} \right ]\right).
\end{array}
\nonumber
\end{eqnarray}
We write the corresponding probability density function compactly as
 \begin{align}
f(\y,\x)=\sum_{k,l}p_k q_l f^{(k,l)}(\y,\x),\nonumber
 \end{align}
where $f^{(k,l)}(\y,\x)$ is a Gaussian density with mean
 \begin{align}
\left [ \begin{array}{c} \H\u^{(k)}_{\x} +\u^{(l)}_{\n}\\ \u^{(k)}_{\x} \end{array} \right ]=\left [ \begin{array}{c}\u^{(k,l)}_{\y}\\ \u^{(k)}_{\x} \end{array} \right ],\nonumber
 \end{align}
and covariance
 \begin{align}
 \left [ \begin{array}{cc} \H\C^{(k)}_{\x\x}\H^T +\C^{(l)}_{\n\n}& \H\C^{(k)}_{\x\x} \\ \C^{(k)}_{\x\x}\H^T & \C^{(k)}_{\x\x} \end{array} \right ]= \left [ \begin{array}{cc} \C^{(k,l)}_{\y\y}& \C^{(k)}_{\y\x} \\ \C^{(k)}_{\x\y} & \C^{(k)}_{\x\x} \end{array} \right ]\nonumber.
 \end{align}
Using Proposition \ref{Marginal distribution of a GM distribution} of the appendix, the marginal density for $\y$ is
\begin{align}
f(\y)=\sum_{k,l}p_k q_l f^{(k,l)}(\y)\label{pdfy},
\end{align}
where $f^{(k,l)}(\y)$ is a Gaussian density with mean $\u^{(k,l)}_{\y}$ and covariance $\C^{(k,l)}_{\y\y}$. That is
\begin{align}
f^{(k,l)}(\y) = \mathcal{N} \left( \y; \u^{(k,l)}_{\y}, \C^{(k,l)}_{\y\y} \right). 
\label{eq:Pdf_FOR_f^{(k,l)}(y)}
\end{align}
The posterior density follows from Bayes' law as
\begin{align}
f(\x|\y) &=\frac{f(\y,\x)}{f(\y)} = \frac{\sum_{k,l}p_k q_l f^{(k,l)}(\y,\x)}{\sum_{r,s}p_r q_s f^{(r,s)}(\y)}\nonumber \\
&=\frac{\sum_{k,l}p_k q_l f^{(k,l)}(\y)f^{(k,l)}(\x|\y)}{\sum_{r,s}p_r q_s f^{(r,s)}(\y)}\nonumber \\
&=\sum_{k,l}\alpha^{(k,l)}(\y)f^{(k,l)}(\x|\y)\label{post},
\end{align}
where
\begin{align}
\alpha^{(k,l)}(\y)=\frac{p_k q_l f^{(k,l)}(\y)}{\sum_{r,s}p_r q_s f^{(r,s)}(\y)}.\label{alfa}
\end{align}
The weight, $\alpha^{(k,l)}(\y)$, can be seen as the joint probability of $\x$ originating from component $k$, and $\n$ originating from component $l$, given the observation $\y$. 
Note that these weights are non-linear in the observation $\y$, and satisfy $\alpha^{(k,l)}(\y) \geq 0$ and $\sum_{k,l}\alpha^{(k,l)}(\y) = 1$.
In (\ref{post}), $f^{(k,l)}(\x|\y)$ is a conditional density of a multivariate Gaussian, $f^{(k,l)}(\y,\x)$. In that case, $f^{(k,l)}(\x|\y)$ is known to be Gaussian (see e.g. Theorem 10.2 of \cite{151045}) with mean
\begin{eqnarray}
\u^{(k,l)}_{\x|\y} & = &  \u^{(k)}_{\x} + \C^{(k)}_{\x\y} \C^{-(k,l)}_{\y\y} \left( \y - \u^{(k,l)}_{\y} \right) \label{m1} \\
& = & \u^{(k)}_{\x} + \C^{(k)}_{\x\x}\H^T \left( \H\C^{(k)}_{\x\x}\H^T  + \C^{(l)}_{\n\n} \right)^{-1} \nonumber\\
& & \hspace{1cm} \left(\y-\H\u^{(k)}_{\x} -\u^{(l)}_{\n}\right),
\label{comppostmean}
\end{eqnarray}
and covariance
 \begin{align}
\C^{(k,l)}_{\x|\y} & =\C^{(k)}_{\x\x}- \C^{(k)}_{\x\y}\C^{-(k,l)}_{\y\y}\C^{(k)}_{\y\x} \label{comppostvar1}\\
& = \C^{(k)}_{\x\x}-\C^{(k)}_{\x\x}\H^T \left(\H\C^{(k)}_{\x\x}\H^T +\C^{(l)}_{\n\n}\right)^{-1}\H \C^{(k)}_{\x\x} \label{comppostvar},
\end{align}
respectively. Here, and later, $\C^{-(k,l)}_{\y\y}$ is short for $\left({\C^{(k,l)}_{\y\y}}\right)^{-1}$. 
The posterior density $f(\x|\y)$ of~(\ref{post}) is clearly GM distributed. 
By Proposition \ref{Expectation of a mixture.} of the appendix, its mean is
\begin{align}
\u_{\x|\y} = E \left\{ \x|\y \right\} = \sum_{k,l}\alpha^{(k,l)}(\y)\u^{(k,l)}_{\x|\y},\label{mmseest}
\end{align}
and, by Proposition \ref{Covariance matrix of a mixture.} of the appendix, the covariance is
\begin{eqnarray}
\begin{array}{rl}
\C_{\x|\y} & = E \left\{ (\x - E\{ \x \}) (\x - E\{ \x \})^T   |\y \right\}  \\ & =\sum_{k,l}\alpha^{(k,l)}(\y)\left(\C^{(k,l)}_{\x|\y} + \u^{(k,l)}_{\x|\y}{\u^{(k,l)}_{\x|\y}}^{T}\right) \\
& \hspace{2cm} - \u_{\x|\y}{\u_{\x|\y}}^{T}.
\end{array}
\label{postcov}
\end{eqnarray}
For the special case when $\left|\mathcal{K}\right|=\left|\mathcal{L}\right|=1$ (Gaussian input and Gaussian noise), the posterior density, $f(\x|\y)$, is purely Gaussian. Then the mean (\ref{mmseest}) reduces to
\begin{align}
&\u_{\x|\y}= E \left\{ \x|\y \right\} = \u^{(1,1)}_{\x|\y} \label{specialmean}
\end{align}
and the covariance (\ref{postcov}) reduces to
\begin{eqnarray}
\C_{\x|\y} =\C^{(1,1)}_{\x|\y}\label{specialcov}.
\end{eqnarray}

\subsection{The MMSE estimator}
The MMSE estimator corresponds to the posterior mean, given in (\ref{mmseest}), that is 
\begin{align}
\hat{\x}_{\text{MMSE}}= \u_{\x|\y}. \nonumber
\end{align}
Inserting (\ref{comppostmean}) into (\ref{mmseest}) proves (\ref{eq:GM_Density_MMSE_Estimator}) in Theorem~\ref{theo:GM_estimator}.
In the special case when $\left|\mathcal{K}\right|=\left|\mathcal{L}\right|=1$ and $f(\x|\y)$ is Gaussian, we note from (\ref{specialmean}) and (\ref{comppostmean}) that this estimator is linear in $\y$, and from (\ref{specialcov}) and (\ref{comppostvar}) that posterior covariance matrix does not depend on $\y$. The latter property makes it easy to characterize the MSE of the estimator when $f(\x|\y)$ is Gaussian.

In the general case, when $f(\x|\y)$ is a multi-component GM, the MMSE estimator (\ref{mmseest}) is non-linear in the observed data $\y$, because of the data dependent weights $\alpha^{(k,l)}(\y)$. Furthermore, because the posterior covariance $\C_{\x|\y}$ of (\ref{postcov}) depends on the observation $\y$, the MSE becomes considerably more difficult to analyze, as we find in section \ref{sec:Error_Analysis_of_the_MMSE_Estimator}.

\subsection{The maximum a posteriori (MAP) estimator}
Although this paper is not about MAP estimation, we mention very briefly that the map estimator can be found (which is perhaps not entirely evident when the distribution is multi modal).
The MAP estimate for $\x$ is
\begin{eqnarray}
\def\argmax{\mathop{\rm arg \, max}}
\hat{\x}_{\text{MAP}}= \argmax_{ \x } f(\x|\y). \nonumber
\end{eqnarray}

Thus $\hat{\x}_{\text{MAP}}$ corresponds to the mode of $f(\x|\y)$. In the special case when $\left|\mathcal{K}\right|=\left|\mathcal{L}\right|=1$ and $f(\x|\y)$ is Gaussian, the MAP and MMSE estimates for $\x$ coincide, because the mode coincides with the mean. In the general case however, when $f(\x|\y)$ is given by (\ref{post}), the posterior is a multi modal GM distribution. The mode of such a distribution cannot be expected to coincide with its mean. A procedure for finding the MAP estimate, is to find all the modes, and identify the one with the largest probability mass. Finding the modes of a GM distribution, is a problem which has been well described and solved in \cite{888716}. Therefore we do not discuss it further here.

\section{Error Analysis of the MMSE Estimator}
\label{sec:Error_Analysis_of_the_MMSE_Estimator}

\subsection{Mean Square Error}\label{Mean square error}
For a given observation $\y$, the MSE of the estimator in (\ref{mmseest}) can be determined by the trace of $\C_{\x|\y}$ of (\ref{postcov}). 
Our main interest is not in the MSE for a particular $\y$, but rather the MSE averaged over all $\y$. Said differently, we are interested in the MSE matrix
\begin{eqnarray}
\begin{array}{l}
\M = \int \C_{\x|\y}f(\y)d\y \\
= \displaystyle\int \left( \! \sum_{k,l} \! \alpha^{(k,l)}(\y) \left( \! \C^{(k,l)}_{\x|\y} \! + \! \u^{(k,l)}_{\x|\y}{\u^{(k,l)}_{\x|\y}}^{T} \right) \! - \! \u_{\x|\y}{\u_{\x|\y}}^{T} \! \right) \\ 
\hspace{7cm} f(\y)d\y. \\ \nonumber
\end{array}
\end{eqnarray}
Using (\ref{pdfy}) and (\ref{alfa}), we obtain
\begin{align}
\begin{array}{l}
\M = \displaystyle\sum_{k,l}p_k q_l\int\left(\C^{(k,l)}_{\x|\y}+\u^{(k,l)}_{\x|\y}{\u^{(k,l)}_{\x|\y}}^{T}-\u_{\x|\y}{\u_{\x|\y}}^{T}\right) \\
\hspace{6cm} f^{(k,l)}(\y)d\y.\label{avgcov}
\end{array}
\end{align}
We inspect the above integral term-by-term. The first term of (\ref{avgcov}) is
\begin{align}
\M_1=\sum_{k,l}p_k q_l\int \C^{(k,l)}_{\x|\y}f^{(k,l)}(\y)d\y =\sum_{k,l}p_k q_l\C^{(k,l)}_{\x|\y},\label{first}
\end{align}
where the last equality holds because $\C^{(k,l)}_{\x|\y}$ is not a function of $\y$, as can be seen in (\ref{comppostvar}).
The second term of (\ref{avgcov}) is
\begin{align}
\M_2=\sum_{k,l}p_k q_l\int\u^{(k,l)}_{\x|\y}{\u^{(k,l)}_{\x|\y}}^{T}f^{(k,l)}(\y)d\y \label{sec1}.
\end{align} 
Inserting (\ref{m1}) into (\ref{sec1}), we obtain
\begin{align}
&\M_2\nonumber\\
&=\sum_{k,l}p_k q_l\int\left[\u^{(k)}_{\x} + \C^{(k)}_{\x\y} \C^{-(k,l)}_{\y\y} \left( \y - \u^{(k,l)}_{\y} \right)\right]\nonumber\\
&\hspace{1.3cm}\left[\u^{(k)}_{\x} + \C^{(k)}_{\x\y} \C^{-(k,l)}_{\y\y} \left( \y - \u^{(k,l)}_{\y} \right)\right]^T f^{(k,l)}(\y)d\y \nonumber \\
&=\sum_{k,l}p_k q_l\left(\u^{(k)}_{\x}{\u^{(k)}_{\x}}^{T}+ \C^{(k)}_{\x\y}\C^{-(k,l)}_{\y\y} \C^{(k)}_{\y\x} \right)\nonumber\\
&=\sum_{k,l}p_k q_l\left(\u^{(k)}_{\x}{\u^{(k)}_{\x}}^{T}+ \C^{(k)}_{\x\x}-\C^{(k,l)}_{\x|\y} \right),\label{sec}
\end{align}
where the last equality is obtained using (\ref{comppostvar1}).
The third term of (\ref{avgcov}) is
\begin{align}
\M_3 =-\int \u_{\x|\y}{\u_{\x|\y}}^{T} \sum_{k,l}p_k q_l f^{(k,l)}(\y)d\y. \label{third1}
\end{align}
Note from (\ref{mmseest}) that
\begin{eqnarray}
\begin{array}{l}
\u_{\x|\y}{\u_{\x|\y}}^{T} \\
= \left( \displaystyle\sum_{k,l} \alpha^{(k,l)}(\y) \u^{(k,l)}_{\x|\y} \right) \left( \displaystyle\sum_{r,s}\alpha^{(r,s)}(\y){\u^{(r,s)}_{\x|\y}}^{T} \right) \\
= \frac{ \left( \displaystyle\sum_{k,l} p_k q_l f^{(k,l)}(\y)\u^{(k,l)}_{\x|\y} \right) \left( \displaystyle\sum_{r,s} p_r q_s f^{(r,s)}(\y){\u^{(r,s)}_{\x|\y}}^{T} \right) }{\left( \displaystyle\sum_{v,w}p_v q_w f^{(v,w)}(\y)\right)^2}. \nonumber
\end{array}
\end{eqnarray}
Hence the integral in (\ref{third1}) can be written
\begin{align}
\M_3=-\sum_{k,l,r,s}p_k q_l p_r q_s\int \frac{ f^{(k,l)}(\y) f^{(r,s)}(\y)\u^{(k,l)}_{\x|\y}{\u^{(r,s)}_{\x|\y}}^{T}}{\sum_{v,w}p_v q_w f^{(v,w)}(\y)}d\y. \nonumber
\end{align}
As far as we can see, this integral cannot be solved analytically, meaning that we cannot determine the MSE matrix exactly. Our main interest is in the trace of $\M$, because this corresponds to the MSE: 
\begin{align}
\epsilon^2=\text{Tr}(\M)=\text{Tr}(\M_1)+\text{Tr}(\M_2)+\text{Tr}(\M_3).\nonumber
\end{align}
In the absence of an analytical expression of $\epsilon^2$, we pursue upper and lower bounds, as follows.
From equations (\ref{first}), (\ref{sec1}) and (\ref{third1}), we note that
\begin{align}
&\text{Tr}(\M_1)=\sum_{k,l}p_k q_l\text{Tr}\left(\C^{(k,l)}_{\x|\y}\right),\nonumber\\
&\text{Tr}(\M_2)=\sum_{k,l}p_k q_l\int{\u^{(k,l)}_{\x|\y}}^{T}\u^{(k,l)}_{\x|\y}f^{(k,l)}(\y)d\y,\label{tr2}\\
&\text{Tr}(\M_3)=-\sum_{k,l}p_k q_l\int{\u_{\x|\y}}^{T}\u_{\x|\y}f^{(k,l)}(\y)d\y,\label{tr3}
\end{align}
respectively. Since $p_k \geq 0$, $q_l \geq 0$, $f^{(k,l)}(\y)$ is a PDF, $\C^{(k,l)}_{\x|\y}$ is a covariance matrix, and ${\u^{(k,l)}_{\x|\y}}^{T}\u^{(k,l)}_{\x|\y}$ and ${\u_{\x|\y}}^{T}\u_{\x|\y}$ are inner products, it can be concluded that
\begin{align}
\text{Tr}(\M_1)\geq 0, \text{   }\text{Tr}(\M_2)\geq 0, \text{   } \text{Tr}(\M_3)\leq 0.\label{conditions}
\end{align}
Furthermore, from (\ref{tr2}) and (\ref{tr3}), we note that
\begin{align}
&\text{Tr}(\M_2)+\text{Tr}(\M_3)\nonumber \\
&=\int\sum_{k,l}p_k q_l f^{(k,l)}(\y)\left({\u^{(k,l)}_{\x|\y}}^{T}{\u^{(k,l)}_{\x|\y}}-{\u_{\x|\y}}^{T}{\u_{\x|\y}}\right)d\y\nonumber \\
&=\int\sum_{k,l}\alpha^{(k,l)}(\y)\left({\u^{(k,l)}_{\x|\y}}^{T}{\u^{(k,l)}_{\x|\y}}-{\u_{\x|\y}}^{T}{\u_{\x|\y}}\right)f(\y)d\y\nonumber\\
&=\int\sum_{k,l}\alpha^{(k,l)}(\y)\left({\u^{(k,l)}_{\x|\y}}-{\u_{\x|\y}}\right)^{T}\left({\u^{(k,l)}_{\x|\y}}-{\u_{\x|\y}}\right)f(\y)d\y\nonumber\\
&=\int\sum_{k,l}\alpha^{(k,l)}(\y)\left\|{\u^{(k,l)}_{\x|\y}}-{\u_{\x|\y}}\right\|^{2}_{2}f(\y)d\y\label{trdiff}\\
&\geq 0, \nonumber
\end{align}
where the second equality is obtained by using (\ref{pdfy}) and (\ref{alfa}); the third equality is obtained by using (\ref{mmseest}) and $\sum_{k,l}\alpha^{(k,l)}(\y) = 1$; and the inequality is obtained by using $\alpha^{(k,l)}(\y) \geq 0$. 
This, combined with the conditions (\ref{conditions}), gives the following bounds
\begin{align}
\text{Tr}(\M_1)\leq &\epsilon^2 \leq\text{Tr}(\M_1)+\text{Tr}(\M_2).
\label{bounds}
\end{align}
By appropriate substitutions using (\ref{first}) and (\ref{comppostvar}), one obtains the lower bound in (\ref{eq:GM_Density_MMSE_Error_Bounds}) of Theorem~\ref{theo:GM_estimator}. 

An alternative argument provides an intuition for the bounds in (\ref{bounds}). Imagine that a side information is available in the estimation process such that, for each observation $\y$, a genie tells us which \textit{single} Gaussian component in (\ref{m_mix}) has generated the underlying $\x$, and also which \textit{single} Gaussian component has generated the underlying $\n$. Said differently, for each $\y$, we face the familiar model of Gaussian signal and Gaussian noise. 
Such a genie-aided estimator can be described as a two-stage estimator consisting of (1) a perfect (error free) decision device, followed by (2) a decision dependent Gaussian signal and Gaussian noise MMSE estimator.  
In this (imaginary but very favorable) case, we note that 
\begin {align}
\M=\M_1, \text{ and hence } \epsilon^2 = \text{Tr}(\M_1). \nonumber
\end{align}
Without a genie, we must expect an error of at least $\text{Tr}(\M_1)$. This implies that $\text{Tr}(\M_2)+\text{Tr}(\M_3)\geq 0$. Since $\text{Tr}(\M_3)\leq 0$, we reach the same conclusions as in (\ref{bounds}). In the next section, we show that there exists a tighter upper bound than the one in (\ref{bounds}).

\subsection{Tightening the Upper Bound}
\label{sec:Lowering_the_Upper_Bound}

The upper bound of (\ref{bounds}), $\text{Tr}(\M_1)+\text{Tr}(\M_2)$, can in fact be replaced by a tighter one. This can be seen by invoking the following argument. Instead of using the optimal MMSE estimator in (\ref{mmseest}), we \textit{could} use a linear MMSE (LMMSE) estimator. The LMMSE estimator is given by (see e.g. Theorem 12.1 of \cite{151045}) 
\begin{align}
\hat{\x}=\u_{\x} + \C_{\x\x}\H^T \left( \H\C_{\x\x}\H^T  + \C_{\n\n}\right)^{-1} \left(\y-\H\u_{\x}-\u_{\n}\right),\label{lmmseest}
\end{align}
with corresponding MSE matrix
\begin{align}
\C_{\x\x} - \C_{\x\x}\H^{T}\left(\H\C_{\x\x}\H^{T}+\C_{\n\n}\right)^{-1}\H\C_{\x\x}.\label{lmmsecov}
\end{align} 
Here, $\u_{\x}$ and $\C_{\x\x}$ are the mean and covariance of $\x$, given by (\ref{mx}) and (\ref{covx}) respectively, and $\u_{\n}$ and $\C_{\n\n}$ are the mean and covariance of $\n$, given by (\ref{mn}) and (\ref{covn}) respectively. 

The MSE of the LMMSE estimator is given by the trace of (\ref{lmmsecov}): 
\begin{align}
\epsilon^{2}_{L} &= \text{Tr} \left(\C_{\x\x} - \C_{\x\x}\H^{T}\left(\H\C_{\x\x}\H^{T}+\C_{\n\n}\right)^{-1}\H\C_{\x\x}\right)\nonumber\\
&=\text{Tr}\left(\C_{\x\x}\right)-\sum_{j}\g^{T}_{j}\left(\H\C_{\x\x}\H^{T}+\C_{\n\n}\right)^{-1}\g_{j}\label{covlmmse},
\end{align}
where $\g_{j}$ is the $j$-th column of $\H\C_{\x\x}$. In (\ref{covlmmse}), $\left(\H\C_{\x\x}\H^{T}+\C_{\n\n}\right)^{-1}$ is a positive semidefinite matrix, which implies that
\begin{align}
\epsilon^{2}_{L} \leq\text{Tr}\left(\C_{\x\x}\right).\nonumber
\end{align}

Now, we compare this with $\text{Tr}(\M_1)+\text{Tr}(\M_2)$. Using (\ref{first}) and (\ref{sec}), we may write
\begin{eqnarray}
\begin{array}{l}
\text{Tr}(\M_1)+\text{Tr}(\M_2) \\
=\text{Tr}\left(\sum_{k}p_k \left(\C^{(k)}_{\x\x}+\u^{(k)}_{\x}{\u^{(k)}_{\x}}^{T} \right)\right) \\
\geq \text{Tr}\left(\sum_{k}p_k \left(\C^{(k)}_{\x\x}+\u^{(k)}_{\x}{\u^{(k)}_{\x}}^{T} \right)-\u_x \u^{T}_x\right) \\
=\text{Tr}\left(\C_{\x\x}\right) \\
\geq \epsilon^{2}_{L}\nonumber ,
\end{array}
\end{eqnarray}
where the last equality follows from using Proposition \ref{Covariance matrix of a mixture.} of the appendix.
Since we know that the LMMSE estimator cannot outperform the optimal MMSE estimator, on average, we can replace $\text{Tr}(\M_1)+\text{Tr}(\M_2)$, by the tighter bound $\epsilon^{2}_{L}$. Note that $\epsilon^{2}_{L}$ in (\ref{covlmmse}) corresponds to the upper bound in (\ref{eq:GM_Density_MMSE_Error_Bounds}) of Theorem~\ref{theo:GM_estimator}. 

In summary, the performance of the optimal MMSE estimator in (\ref{mmseest}) is lower bounded by a genie-aided MMSE estimator and upper bounded by the LMMSE estimator.

\subsection{Simple Examples: High and Low SNR Cases}\label{Simple Example: High and Low SNR Cases}
Intuitively, one expects that the MSE approaches its lower bound as the signal-to-noise-ratio,
\begin{align}
\text{SNR}=\frac{E\left\{\left\|\x\right\|^{2}_{2}\right\}}{E\left\{\left\|\n\right\|^{2}_{2}\right\}},\nonumber
\end{align}
 goes to infinity and the upper bound as the SNR goes to zero. We will demonstrate that this is true for a simple, but instructive, example. 
Throughout this example, we assume the noise to be distributed as
\begin{align}
\n\sim a\sum_{l\in\mathcal{L}}q_l \mathcal{N}(\u^{(l)}_{\n},\C^{(l)}_{\n\n})=\sum_{l\in\mathcal{L}}q_l \mathcal{N}(a\u^{(l)}_{\n},a^2\C^{(l)}_{\n\n}),\label{noisedist}
\end{align}
where $a$ is a scalar which can be set to account for any SNR level. Furthermore, we assume that $\H$ is a full rank square matrix. Then (\ref{comppostmean}) can be written
\begin{align}
\u^{(k,l)}_{\x|\y}=\u^{(k)}_{\x} + \C^{(k)}_{\x\x}\H^T \left( \H\C^{(k)}_{\x\x}\H^T  + a^2 \C^{(l)}_{\n\n} \right)^{-1}\nonumber\\
\hspace{4cm} \left(\y-\H\u^{(k)}_{\x}-a\u^{(l)}_{\n}\right).\nonumber
\end{align}

\subsubsection{High SNR}\label{Infinite SNR}
We drive the SNR towards infinity by $\lim a\rightarrow 0$. Then the above reads
\begin{align}
\u^{(k,l)}_{\x|\y}&=\u^{(k)}_{\x} + \C^{(k)}_{\x\x}\H^T \left( \H\C^{(k)}_{\x\x}\H^T  \right)^{-1} \left(\y-\H\u^{(k)}_{\x}\right)\nonumber\\
&=\u^{(k)}_{\x} + \H^{-1}\left(\y-\H\u^{(k)}_{\x}\right)\nonumber \\
&=\H^{-1}\y.\label{resex}
\end{align}
Thus, the component means of the the posterior are all the same. In that case we have $\u_{\x|\y}=\u^{(k,l)}_{\x|\y}$, and from (\ref{trdiff}) it can be verified that $\text{Tr}(\M_2)+\text{Tr}(\M_3)=0$. Hence, the MSE will be determined by $\text{Tr}(\M_1$) only, and by (\ref{bounds}) it therefore reaches the lower bound.
This bound can be found, using (\ref{comppostvar}), which in our case reduces to: 
 \begin{align}
\C^{(k,l)}_{\x|\y} & = \C^{(k)}_{\x\x}-\C^{(k)}_{\x\x}\H^T \left(\H\C^{(k)}_{\x\x}\H^T \right)^{-1}\H \C^{(k)}_{\x\x} \nonumber\\
&=\C^{(k)}_{\x\x}-\H^{-1}\H \C^{(k)}_{\x\x}=0.\nonumber
\end{align}
Inserting this into (\ref{first}), and taking the trace, we find that the lower bound of the MSE is zero.
Note from (\ref{resex}), that the estimator discards all prior knowledge and completely trusts the data. This is expected at infinitely high SNR.

Finally, we remark that the MSE of the LMMSE estimator also will also be zero when the SNR goes to infinity: With $\n$ distributed as in (\ref{noisedist}), the LMMSE estimator in (\ref{lmmseest}) becomes
\begin{align}
\hat{\x}=\u_{\x} + \C_{\x\x}\H^T \left( \H\C_{\x\x}\H^T  + a^2 \C_{\n\n}\right)^{-1}\nonumber\\
\hspace{4cm} \left(\y-\H\u_{\x}-a\u_{\n}\right).\label{lmmseest2}
\end{align}
Taking $\lim a\rightarrow 0$, this simplifies to
\begin{align}
 \hat{\x}=\H^{-1}\y.\nonumber 
\end{align}
But this is the same as (\ref{resex}). Hence, at very high SNR the LMMSE estimator and the optimal MMSE estimator coincide, and therefore have the same performance.

\subsubsection{Low SNR}\label{Zero SNR}
Here, it is convenient to rewrite (\ref{comppostmean}) in an alternative, but equivalent form
\begin{align}
\u^{(k,l)}_{\x|\y}&=&\u^{(k)}_{\x}+\left(\C^{-(k)}_{\x\x}+\H^T \C^{-(l)}_{\n\n}\H\right)^{-1}\H^{T} \C^{-(l)}_{\n\n}\nonumber\\
& & \hspace{1cm} \left(\y-\H\u^{(k)}_{\x} -\u^{(l)}_{\n}\right).\nonumber
\end{align}
With $\n$ distributed as in (\ref{noisedist}), this becomes
\begin{align}
\u^{(k,l)}_{\x|\y}&=&\u^{(k)}_{\x}+\left(\C^{-(k)}_{\x\x}+\frac{1}{a^2}\H^T \C^{-(l)}_{\n\n}\H\right)^{-1}\frac{1}{a^2}\H^{T} \C^{-(l)}_{\n\n}\nonumber\\
& & \hspace{1cm} \left(\y-\H\u^{(k)}_{\x} -a\u^{(l)}_{\n}\right).\nonumber
\end{align}
When driving the SNR very low, by $\lim a\rightarrow \infty$, this reduces to
\begin{align}
\u^{(k,l)}_{\x|\y}=\u^{(k)}_{\x}.\nonumber
\end{align}
Thus, the MMSE estimate for $\x$ is
\begin{align}
\u_{\x|\y}&=\sum_{k,l}\alpha^{(k,l)}(\y)\u^{(k,l)}_{\x|\y}\nonumber\\
&=\frac{\sum_{k,l} p_k q_l f^{(k,l)}(\y)}{\sum_{r,s}p_r q_s f^{(r,s)}(\y)}\u^{(k)}_{\x}\nonumber\\
&=\sum_k p_k \u^{(k)}_{\x}.\label{infsnrmse}
\end{align}
In (\ref{infsnrmse}), the last equality holds because $f^{(k,l)}(\y)$ has covariance 
\begin{align}
\C^{(k,l)}_{\y\y}=\H\C^{(k)}_{\x\x}\H^T +a^2 \C^{(l)}_{\n\n}\nonumber
\end{align}
and when $a\rightarrow\infty$, $f^{(k,l)}(\y)$ approaches a uniform distribution with infinite support. Hence, it approaches a constant which is independent of $\y$, $k$ and $l$, and we may simply disregard it. Note from (\ref{infsnrmse}), that the estimator discards the data and uses only prior information, which is expected at zero SNR.
Now, we turn to the LMMSE estimator (\ref{lmmseest2}), which may be rewritten equivalently as
\begin{align}
\hat{\x}=\u_{\x}+\left(\C^{-1}_{\x\x}+\frac{1}{a^2}\H^T \C^{-1}_{\n\n}\H\right)^{-1}\frac{1}{a^2}\H^{T} \C^{-1}_{\n\n}\nonumber\\
\hspace{4cm}\left(\y-\H\u_{\x}-a\u_{\n}\right).\nonumber
\end{align}
With $\lim a\rightarrow \infty$, this reduces to
\begin{align}
\hat{\x}=\u_{\x}=\sum_k p_k \u^{(k)}_{\x}.\label{ss}
\end{align}
But (\ref{ss}) is equal to (\ref{infsnrmse}). Thus, also at very low SNR, the MMSE estimator and the LMMSE estimator coincide. In that case, the error of the MMSE estimator coincides with $\epsilon^{2}_{L}$ in (\ref{covlmmse}), which corresponds to the upper bound.

In summary, in the asymptotic cases of infinite and zero SNR, the MMSE estimator attains minimum and maximum error respectively. In these extreme cases, one might just as well use the simpler LMMSE estimator, because it performs identically.

\section{Simulation Results}
\label{sec:Simulation_Results}
We have shown that at infinite and zero SNR, the LMMSE estimator is just as good as the MMSE estimator. Now we demonstrate that at more realistic and intermediate SNRs, the MMSE estimator certainly outperforms the LMMSE estimator. We do this using Monte Carlo simulations. 
An \textit{estimate} of $\epsilon^2$ can be obtained by calculating the sample mean of $\left\|\x-\u_{\x|\y}\right\|^{2}_{2}$ from many independent observations. 
The plot in Figure \ref{fig:1} shows the lower bound, $\text{Tr}(\M_1)$, the upper bound  $\epsilon^{2}_{L}$ and an estimate of $\epsilon^2$, all in dB, versus an increasing SNR. The SNR ranges from -10 dB to 50 dB in steps of 1 dB. The following parameters have been used:
\begin{itemize}
	\item $\H=\I$, with $\I$ being $5\times 5$.
	\item $\x$ is GM distributed with $|\mathcal{K}|=4$. The component means are the columns of the following matrix
	 \begin{align}
 \left [ \begin{array}{cccc} 35.381 & -47.087 & 79.522 & -30.903 \\
-20.184 & 0.286 & -51.577 & -5.826\\ 
-6.377 & -68.308 & -17.330 & 3.246 \\
24.419 & 4.400 & -7.422 & -101.586\\ 
38.891 & 1.195 & 9.282126 & -0.047508 \end{array} \right ]. \nonumber
 \end{align}
These columns have simply been drawn independently from $\mathcal{N}(0,\sqrt{1000}\I)$. We use component covariance matrices $\C^{(k)}_{\x\x}=\I$, and uniform component probabilities $p_k=1/|\mathcal{K}|=1/4$.
	\item Gaussian noise: $\n\sim\mathcal{N}(0,\beta\I)$. Proper adjustment of $\beta$ provides the required SNRs.
\end{itemize}

\begin{figure}
\includegraphics[scale=0.32]{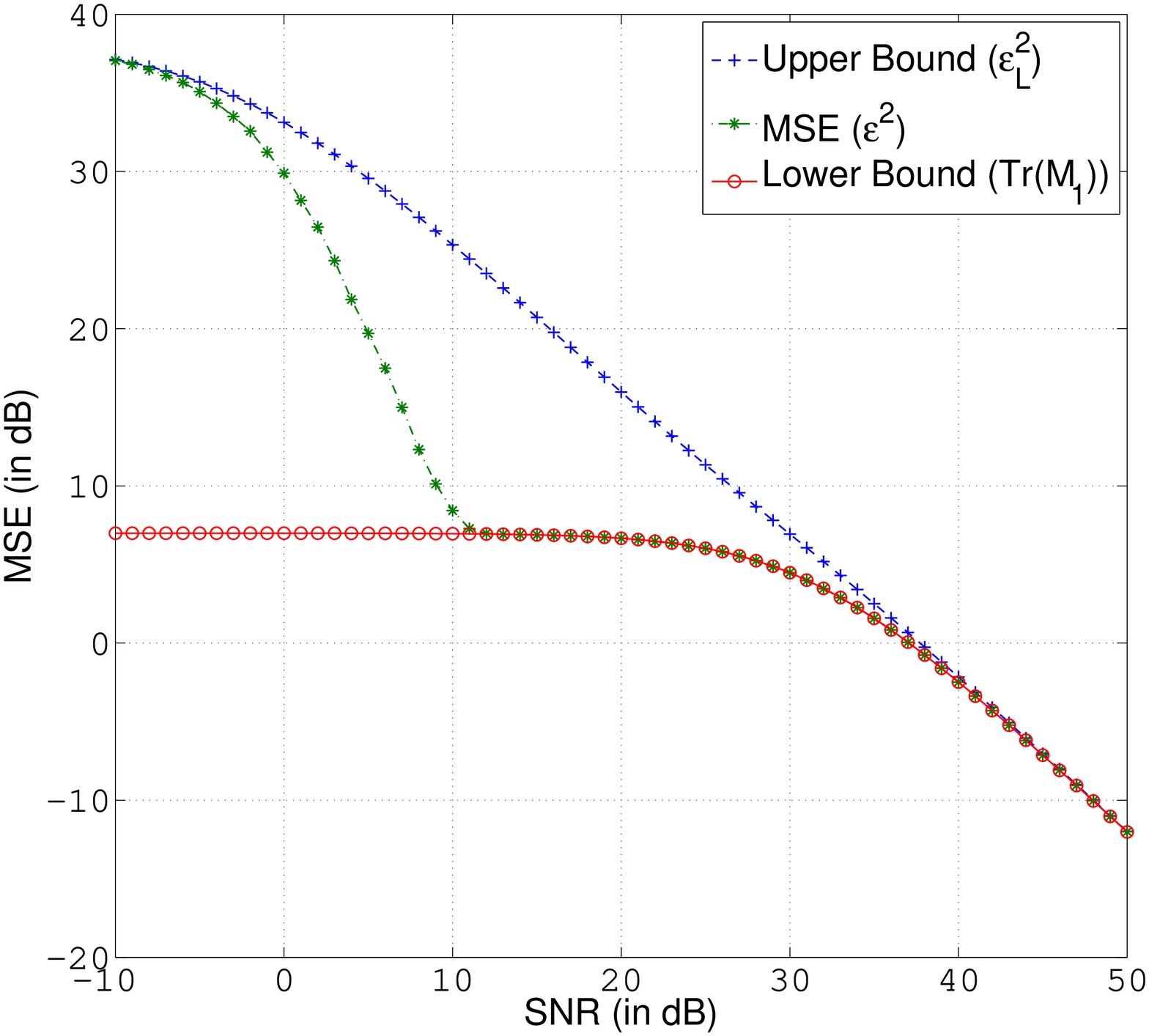}
\vspace{-0.3cm}
\caption{Estimate of the Bayesian MSE $\epsilon^2$, together with its upper and lower bounds. $\H=\I$ and uniform $p_k$.}
\label{fig:1}
\end{figure}

The estimated MSE ($\epsilon^{2}$) is obtained by averaging over 50000 independent $\y$'s for each SNR value. 
One observes that Figure (\ref{fig:1}) is in line with our findings in section \ref{Simple Example: High and Low SNR Cases}: At low SNR, the MSE approaches its upper bound, and at high SNR it approaches the lower, both of which coincide with the MSE of the LMMSE estimator. Note however, that at intermediate SNR values, the optimal MMSE estimator outperforms the LMMSE estimator (the upper bound) quite substantially - and most impressively, for finite and quite modest SNRs (approximately 10 dB and larger), the MMSE estimator performs as if it was helped by a genie. 

Without showing further plots, we remark that in the case when the component means of $\x$ have less variance (are less scattered) than in our example, then $\x$ is in principle more 'Gaussian', and the MSE will be closer to the upper bound for all SNR values. Similarly, when the component means have larger variance (are more scattered) than in our example, then $\x$ becomes more distinctly GM distributed, and the MSE starts to drop from the upper bound at even lower SNR values.

For the interested reader, the MATLAB code which produced the plot in Figure \ref{fig:1}, can be downloaded from: http://sites.google.com/site/saikatchatt/softwares.

\section{Conclusion}
\label{sec:Conclusions}

We have provided the necessary theoretical foundation and derived the MMSE estimator from the Bayesian linear model, when both the noise and the signal have GM distributions. Furthermore, we have shown that the MSE of this estimator cannot be determined in closed form, but that it can be upper bounded by an LMMSE estimator, and lower bounded by a genie aided MMSE estimator. Monte Carlo simulations confirm the bounds, and show that the difference in performance between the optimal MMSE estimator and the LMMSE estimator may be substantial.

\section{Acknowledgments}
John T. Fl{\aa}m's work is supported by the Research Council of Norway under the
NORDITE/VERDIKT program, Project CROPS2 (Grant 181530/S10). Saikat Chatterjee is funded in part by VINNOVA, the Strategic Research Area project RICSNET, and EU FP7 FeedNetBack. Kimmo Kansanen has received funding from the European Communitys Seventh Framework Program (FP7/2007-2013) under grant agreement nr 216076 (FP7-SENDORA).

\section{Appendix: Transforms of GM distributed random vectors}
\label{appendix}
In the literature, mixture distributions are often characterized by a convex combination of probability density functions, see e.g \cite{mclachlan200001}, \cite{citeulike:5193164}. Since not all random variables can be characterized by a probability density function (not all probability measures have a density \cite{ito}), the results presented in this appendix do not rely on probability densities. The results are obtained using distributions (alias measures) and characteristic functions, both of which always exist.

Propositions \ref{Expectation of a mixture.} and \ref{Covariance matrix of a mixture.} can be found in similar form in \cite{888716}. The other propositions, may well exist in the literature, but we have not been able to find them. Since much of our work depends on these propositions, it is natural to include them.

\begin{definition}\label{Finite Mixture distribution}\textbf{Finite Mixture distribution.}\\
Let $\mathcal{K}$ be a finite index set. For each $k\in\mathcal{K}$, let $p_k$ be the probability of drawing index $k$ from $\mathcal{K}$, and let $P_k$ be a probability distribution (or measure) on a Euclidean
(finite-dimensional vector) space $\mathbb{X}$.
Then, the convex combination
\begin{equation}
P=\sum_{k\in\mathcal{K}}p_k P_{k}\label{mixture}%
\end{equation}
also defines a probability distribution on $\mathbb{X}$. We call (\ref{mixture}) as a finite mixture distribution on $\mathbb{X}$.  
\end{definition}

\begin{definition}\label{Finite Gaussian Mixture distribution}\textbf{Gaussian Mixture (GM) distribution.}\\
When all component measures $\left\{ P_k \right\}$ are Gaussian, we call (\ref{mixture}) as a (finite) Gaussian mixture (GM) distribution.
We indicate that a random variable $\x$ is GM distributed by writing 
\begin{align}
\x\sim \sum_{k}p_k \mathcal{N}(\u^{(k)}_{\x},\C^{(k)}_{\x\x}),\nonumber
\end{align}
where it is implicit that $k$ belongs to a finite index set.
\end{definition}

In the following, $\x$ denotes a vector in the sample space $\mathbb{X}$. We define all vectors as column vectors, and assume all samples spaces to be continuous. 

\begin{proposition}\label{Expectation of a mixture.}\textbf{Mean of a mixture.}\\ Suppose $P_k$ has finite mean
\begin{align}
\u^{(k)}_{\x} =\int_{\x\in\mathbb{X}} \x P_{k}(d\x).\nonumber
\end{align}
Then the mixture distribution of (\ref{mixture}) has mean
\begin{align}
\u_{\x}=\sum_{k\in\mathcal{K}}p_k \u^{(k)}_{\x}.\nonumber
\end{align}
\end{proposition}

\begin{proof}
\begin{eqnarray}
\begin{array}{rcl}
\u_{\x} & = & \displaystyle\int_{\x\in\mathbb{X}} \x \sum_{k\in\mathcal{K}}p_k P_{k}(d\x) \\
& = & \displaystyle\sum_{k\in\mathcal{K}}p_k \int_{\x\in\mathbb{X}}\x P_{k}(d\x) \\
& = & \displaystyle\sum^{K}_{k=1}p_k \u^{(k)}_{\x} \nonumber .
\end{array}
\end{eqnarray}
\end{proof}

\begin{proposition}\label{Covariance matrix of a mixture.}\textbf{Covariance of a mixture.}\\
Suppose $P_k$ has the finite mean $\u^{(k)}_{\x}$, and all elements of the covariance matrix
\begin{align}
\C^{(k)}_{\x\x}:=\int_{\x\in\mathbb{X}}(\x-\u^{(k)}_{\x})(\x-\u^{(k)}_{\x})^{T}P_{k}(d\x)\nonumber
\end{align}
have finite magnitudes.
Then, the covariance of the mixture distribution (\ref{mixture}) is
\begin{align}
\C_{\x\x} =\sum_{k\in\mathcal{K}}p_k \left(\C^{(k)}_{\x\x}+\u^{(k)}_{\x}{\u^{(k)}_{\x}}^{T}\right)-\u_{\x}{\u_{\x}}^{T}.
\nonumber
\end{align}
\end{proposition}

\begin{proof}
We use the fact that $\C_{\x\x} = E(\mathbf{xx}^{T})-E(\x)E(\x)^{T}$ always holds. Thus
\begin{eqnarray}
\begin{array}{rcl}
\C_{\x\x} & = & \displaystyle\int_{\x\in\mathbb{X}} \mathbf{xx}^{T}\sum_{k\in\mathcal{K}}p_k P_{k}(d\x)-\u_{\x}\u_{\x}^{T}  \\
& = & \displaystyle\sum_{k\in\mathcal{K}}p_k \int_{\x\in\mathbb{X}} \mathbf{xx}^{T}P_{k}(d\x)-\u_{\x}{\u_{\x}}^{T}  \\
& = & \displaystyle\sum_{k\in\mathcal{K}}p_k \left(\C^{(k)}_{\x\x} + \u^{(k)}_{\x}{\u^{(k)}_{\x}}^{T}\right) - \u_{\x}{\u_{\x}}^{T}. \nonumber
\end{array}
\end{eqnarray}
\end{proof}

\begin{proposition}\label{Characteristic function of GM distributed RV.}\textbf{Characteristic function of a GM distributed random vector.}\\
Let $\x\sim \sum_{k}p_k \mathcal{N}(\u^{(k)}_{\x},\C^{(k)}_{\x\x})$. Then the characteristic function of $\x$ is (see e.g. \cite{GVK025286692})
\begin{align}
\phi(\t)=\sum_{k}p_k e^{i\t^{T}\u^{(k)}_{\x}-\frac{1}{2}\t^{T}\C^{(k)}_{\x\x}\t}.\nonumber
\end{align}
for any real vector $\t$.
\end{proposition}

\begin{proof}
For any real vector $\t$, the characteristic function for $\x\sim \mathcal{N}(\u_{\x},\C_{\x\x})$ is
\begin{align}
\phi(\t)=\int e^{i\t^{T}\x} P(d\x)=e^{i\t^{T}\u_{\x}-\frac{1}{2}\t^{T}\C_{\x\x}\t} \nonumber 
\end{align}
where $P=\mathcal{N}(\u_{\x},\C_{\x\x})$. Now, if $\x\sim \sum_{k}p_k \mathcal{N}(\u^{(k)}_{\x},\C^{(k)}_{\x\x})$, then the characteristic function is
\begin{eqnarray}
\begin{array}{rcl}
\phi(\t) & = & E\left(e^{i\t^{T}\x}\right) \\
& = & \displaystyle\int e^{i\t^{T}\x}\sum_{k}p_k P_{k}(d\x) \\
& = & \displaystyle\sum_{k}p_k\int e^{i\t^{T}\x} P_{k}(d\x) \\
& = & \displaystyle\sum_{k}p_k e^{i\t^{T}\u^{(k)}_{\x}-\frac{1}{2}\t^{T}\C^{(k)}_{\x\x}\t}. \nonumber
\end{array}
\end{eqnarray}
\end{proof}

\begin{proposition}\label{Joint distribution of independent GM distributed random vectors}
\textbf{Joint distribution of independent GM distributed random vectors}.\\ 
Let $\x\sim \sum_{k}p_k \mathcal{N}(\u^{(k)}_{\x},\C^{(k)}_{\x\x})$ and $\y\sim \sum_{r}q_r\mathcal{N}(\u^{(r)}_{\y},\C^{(r)}_{\y\y})$, where and $\x$ and $\y$ are mutually independent. Then $\x$ and $\y$ are jointly GM distributed as
\begin{align}
 \left [ \begin{array}{c} \x \\ \y \end{array} \right ] \sim \sum_{k,r}p_k q_r \mathcal{N}\left(\left [ \begin{array}{c} \u^{(k)}_{\x} \\ \u^{(r)}_{\y} \end{array} \right ],\left [ \begin{array}{cc} \C^{(k)}_{\x\x} & 0 \\ 0 & \C^{(r)}_{\y\y} \end{array} \right ]\right).\nonumber
 \end{align}
\end{proposition}

\begin{proof}
By Proposition \ref{Characteristic function of GM distributed RV.}, the characteristic functions of $\x$ and $\y$ are
\begin{align}
\phi_{\x}(\t)=\sum_{k}p_k e^{i\t^{T}\u^{(k)}_{\x}-\frac{1}{2}\t^{T}\C^{(k)}_{\x\x}\t} \nonumber
\end{align}
and
\begin{align}
\phi_{\y}(\s)=\sum_{r}q_r e^{i\s^{T}\u^{(r)}_{\y}-\frac{1}{2}\s^{T}\C^{(r)}_{\y\y}\s} \nonumber
\end{align}
respectively. Because of the independence, the characteristic function of the joint random vector $[\x^{T} \y^{T}]^{T}$ is 
\begin{eqnarray}
\begin{array}{l}
\phi_{\x,\y}\left( \left [ \begin{array}{c} \t \\ \s \end{array} \right ]\right)  \\
= \phi_{\x}(\t)\phi_{\y}(\s) \\
= \displaystyle\sum_{k,r}p_k q_r e^{i\left(\t^{T}\u^{(k)}_{\x}+\s^{T}\u^{(r)}_{\y}\right)-\frac{1}{2}\left(\t^{T}\C^{(k)}_{\x\x}\t+\s^{T}\C^{(r)}_{\y\y}\s\right)} \\
= \displaystyle\sum_{k,r}p_k q_r \exp\left(i\left[\t^{T} \s^{T}\right]\left [ \begin{array}{c} \u^{(k)}_{\x} \\ \u^{(r)}_{\y} \end{array} \right ] \right.  \\
\left. \hspace{3cm} -\frac{1}{2}\left[\t^{T} \s^{T}\right]\left [ \begin{array}{cc} \C^{(k)}_{\x\x} & 0 \\ 0 & \C^{(r)}_{\y\y} \end{array} \right ]\left [ \begin{array}{c} \t \\ \s \end{array} \right ]\right)\nonumber
\end{array}
\end{eqnarray}
for any real vector $[\t^{T} \s^{T}]^{T}$. 
\end{proof}

\begin{proposition}\label{Affine transform of a GM distributed random vector.}
\textbf{Affine transform of a GM distributed random vector.}\\ Let $\mathbf{y=Dx+a}$ where $\x\sim \sum_{k}p_k \mathcal{N}(\u^{(k)}_{\x},\C^{(k)}_{\x\x})$. Then 
\begin{align} 
\y\sim \sum_{k}p_k\mathcal{N}(\D\u^{(k)}_{\x}+\a,\D{\C^{(k)}_{\x\x}}\D^{T})\nonumber.
\end{align}
\end{proposition}
\begin{proof}
\begin{eqnarray}
\begin{array}{rcl}
\phi_{\y}(\t) & = & E \left( e^{i\t^{T} \left( \D\x + \a \right)} \right) \\
& = & e^{i\t^{T}\a}E\left(e^{i\left(\D^{T}\t\right)^{T}\x}\right) \\
& = & e^{i\t^{T}\a} \displaystyle\sum_{k} p_k e^{i \left( \D^{T}\t \right)^{T} \u^{(k)}_{\x} - \frac{1}{2} \left( \D^{T}\t \right)^{T} \C^{(k)}_{\x\x} \left( \D^{T}\t \right)} \\
& = & \displaystyle\sum_{k}p_k e^{i\t^{T}\left(\D\u^{(k)}_{\x}+\a\right)-\frac{1}{2}\t^{T}\D\C^{(k)}_{\x\x}\D^{T}\t}.  \nonumber
\end{array}
\end{eqnarray}
\end{proof}
 
\begin{proposition}\label{Marginal distribution of a GM distribution}
\textbf{Marginal distribution of a GM distribution}.\\ Let $\x\sim \sum_{k}p_k \mathcal{N}(\u^{(k)}_{\x},\C^{(k)}_{\x\x})$. Partition $\x$ into two sub vectors such that
\begin{align} 
&\x=\left [ \begin{array}{c} \x_1 \\ \x_2 \end{array} \right ], \text{ }\u^{(k)}_{\x}=\left [ \begin{array}{c} \u^{(k)}_{\x_1} \\ \u^{(k)}_{\x_2} \end{array} \right ]\text{ and }\nonumber\\
 &\C^{(k)}_{\x\x}=\left [ \begin{array}{cc} \C^{(k)}_{\x_{1}\x_{1}} & \C^{(k)}_{\x_{1}\x_{2}} \\ \C^{(k)}_{\x_{2}\x_{1}} & \C^{(k)}_{\x_{2}\x_{2}} \end{array} \right ]\nonumber. 
\end{align}
Then the marginal distribution for $\x_1$ is $\sum_{k}p_k \mathcal{N}(\u^{(k)}_{\x_1},\C^{(k)}_{\x_{1}\x_{1}})$.
\end{proposition}

\begin{proof}
Without loss of generality, assume that $\x_1$ contains the $p$ first elements of $\x$. Let 
\begin{align}
\D=\left [ \begin{array}{cc} \I_p & 0 \\ 0 & 0 \end{array} \right ]\nonumber.
\end{align}
Then $\x_1=\D\x$, and by Proposition \ref{Affine transform of a GM distributed random vector.} the statement is proved.
\end{proof}

\bibliographystyle{IEEEbib}
\bibliography{strings}

\begin{thebibliography}{10}

\bibitem{151045}
Steven~M. Kay,
\newblock {\em Fundamentals of statistical signal processing: Estimation
  theory},
\newblock Prentice-Hall, Inc., Upper Saddle River, NJ, USA, 1993.

\bibitem{Sorenson1971465}
H.W. Sorenson and D.L. Alspach,
\newblock ``Recursive bayesian estimation using gaussian sums,''
\newblock {\em Automatica}, vol. 7, no. 4, pp. 465 -- 479, 1971.

\bibitem{928914}
J.~Samuelsson and P.~Hedelin,
\newblock ``Recursive coding of spectrum parameters,''
\newblock {\em Speech and Audio Processing, IEEE Transactions on}, vol. 9, no.
  5, pp. 492 --503, jul 2001.

\bibitem{4808405}
D.~Persson and T.~Eriksson,
\newblock ``Mixture model- and least squares-based packet video error
  concealment,''
\newblock {\em Image Processing, IEEE Transactions on}, vol. 18, no. 5, pp.
  1048 --1054, may 2009.

\bibitem{1597257}
A.D. Subramaniam, W.R. Gardner, and B.D. Rao,
\newblock ``Low-complexity source coding using gaussian mixture models, lattice
  vector quantization, and recursive coding with application to speech spectrum
  quantization,''
\newblock {\em Audio, Speech, and Language Processing, IEEE Transactions on},
  vol. 14, no. 2, pp. 524 -- 532, march 2006.

\bibitem{Flamicassp}
J.T. Fl{\aa}m, J.~Jald{\'e}n, and S.~Chatterjee,
\newblock ``{G}aussian {M}ixture {M}odeling for {S}ource {L}ocalization,''
\newblock in {\em ICASSP 2011}, Prague - Chech Republic, May 22-27 2011.

\bibitem{4518754}
A.~Kundu, S.~Chatterjee, A.~Sreenivasa~Murthy, and T.V. Sreenivas,
\newblock ``{GMM} based {B}ayesian {A}pproach to {S}peech {E}nhancement in
  {S}ignal / {T}ransform {D}omain,''
\newblock in {\em ICASSP 2008}, april 4 2008, pp. 4893 --4896.

\bibitem{kundusaikat}
A.~Kundu, S.~Chatterjee, and T.V. Sreenivas,
\newblock ``{S}ubspace {B}ased {S}peech {E}nhancement {U}sing {G}aussian
  {M}ixture {M}odel,''
\newblock in {\em Interspeech 2008}, Brisbane, Australia, september 2008.

\bibitem{kundusaikatsr}
A.~Kundu, S.~Chatterjee, and T.V. Sreenivas,
\newblock ``{S}peech {E}nhancement using {I}tra-{F}rame {D}ependency in {DCT}
  {D}omain,''
\newblock in {\em 16th European Signal Processing Conference (EUSIPCO 2008)},
  Lausanne, Switzerland, augaust 25-29 2008.

\bibitem{1628628}
I.~Bilik and J.~Tabrikian,
\newblock ``Optimal recursive filtering using gaussian mixture model,''
\newblock in {\em Statistical Signal Processing, 2005 IEEE/SP 13th Workshop
  on}, july 2005, pp. 399 --404.

\bibitem{888716}
M.A. Carreira-Perpinan,
\newblock ``Mode-finding for mixtures of gaussian distributions,''
\newblock {\em Pattern Analysis and Machine Intelligence, IEEE Transactions
  on}, vol. 22, no. 11, pp. 1318 -- 1323, Nov. 2000.

\bibitem{mclachlan200001}
Geoffrey~J. McLachlan and David Peel,
\newblock {\em Finite {M}ixture {M}odels}, vol. 299 of {\em Probability and
  Statistics -- Applied Probability and Statistics Section},
\newblock Wiley, New York, 2000.

\bibitem{citeulike:5193164}
Sylvia Fr\"{u}hwirth-Schnatter,
\newblock {\em {Finite Mixture and Markov Switching Models (Springer Series in
  Statistics)}},
\newblock Springer, 1 edition, August 2006.

\bibitem{ito}
Kiyosi Ito,
\newblock {\em Introduction to probability theory},
\newblock Cambridge University Press, the {E}nglish-language edition edition,
  1984.

\bibitem{GVK025286692}
{Theodore Wilbur} Anderson,
\newblock {\em An {I}ntroduction to {M}ultivariate {S}tatistical {A}nalysis},
\newblock Wiley, New York [u.a.], 2. ed edition, 1984.

\end{thebibliography}
  
\end{document}